\documentclass[a4paper]{article}

\usepackage{amsmath}
\usepackage{amssymb}
\usepackage{amsthm}
\usepackage{hyperref}
\usepackage[UKenglish]{isodate}
\usepackage{mathtools}
\usepackage{rotating}
\usepackage{tikz}

\theoremstyle{remark}

\theoremstyle{definition}
\newtheorem{definition}{Definition}[section]
\newtheorem{example}[definition]{Example}

\theoremstyle{theorem}
\newtheorem{lemma}[definition]{Lemma}

\newtheorem{theorem}[definition]{Theorem}
\newtheorem{corollary}[definition]{Corollary}
\newtheorem{conjecture}[definition]{Conjecture}

\newcommand{\duall}[1]{{{}^\vee {#1}}}

\newcommand{\id}{\operatorname{id}}
\newcommand{\inv}{{-1}}
\newcommand{\iso}{\cong}
\newcommand{\munit}{\mathcal{I}}
\newcommand{\op}{\mathrm{op}}
\newcommand{\tensor}{\otimes}

\newcommand{\Cat}{\textsc{Cat}}
\newcommand{\Rep}{\textsc{Rep}}
\newcommand{\Set}{\textsc{Set}}
\newcommand{\Vect}{\textsc{Vect}}

\newcommand{\C}{\mathcal{C}}
\newcommand{\J}{\mathcal{J}}
\newcommand{\M}{\mathcal{M}}
\newcommand{\X}{\mathcal{X}}

\tikzstyle{compactarrow}=[rotate=30, draw, fill=white]
\tikzstyle{compactobject}=[rotate=30]
\tikzstyle{equal}=[double equal sign distance]

\begin{document}

\title{Semidirect Products of Monoidal Categories}

\author{Ben Fuller}

\maketitle

\begin{abstract}
	We introduce semidirect products of skew monoidal categories as a categorification of semidirect products of monoids (or, perhaps more familiarly, of groups).
	We also discuss how this construction interacts with monoidal, autonomous and closed monoidal structures.
	We end by producing some examples of monoidal categories which are left closed but not right closed.
\end{abstract}

\tableofcontents

\section{Introduction} \label{section:introduction}

In this paper, we introduce semidirect products of skew monoidal categories as a categorification of semidirect products of monoids (or, perhaps more familiarly, of groups).
Skew monoidal categories are like ordinary monoidal categories, except that the associator and unitors are not required to be invertible.
They were first introduced by Szlach\'anyi \cite{Szlachányi2012} as a method of describing right bialgebroids over a ring,
and have subsequently been studied by others, such as Lack and Street (e.g.\ \cite{LackStreet2012} \cite{LackStreet2014}), in other contexts.
In this case, it turns out that a trivial semidirect product skew monoidal category of the form $\{\star\} \ltimes \C$ results in what is referred to by Szlach\'anyi as the skew monoidal category `corepresented' by a lax monoidal comonad $T \colon \C \rightarrow \C$.
This skew monoidal category is obtained by defining a new tensor product, $\hat{\tensor}$, in terms of the old tensor product, $\tensor$, as follows.
\[ A \mathbin{\hat{\tensor}} B = T (A) \tensor B \]

We also discuss how this construction interacts with monoidal, autonomous and closed monoidal structures.
One example we describe is the following semidirect product monoidal category, $\Set^\op \ltimes \Set$.
An object, denoted $\langle X, C \rangle$, is a pair of sets, $X$ and $C$.
The tensor product is defined as follows.
\[ \langle X, A \rangle \tensor \langle Y, B \rangle = \langle X \times Y, [Y, A] \times B \rangle \]

We end by producing some examples of monoidal categories which are left closed but not right closed.
One such example is the following semidirect product left closed monoidal category, $\Set \ltimes \Set$.
An object, denoted $\langle X, C \rangle$, is a pair of sets, $X$ and $C$.
The tensor product is defined as follows.
\[ \langle X, A \rangle \tensor \langle Y, B \rangle = \langle X \times Y, (A \times Y) + B \rangle \]
The internal hom is defined as follows.
\[ [\langle X, A \rangle, \langle Z, C \rangle] = \langle [X, Z] \times [A, C], C \rangle \]

\section{Background} \label{section:background}

In this section, we will familiarise ourselves with some basic concepts and establish some of the notation used throughout this paper.

We will now give a definition of semidirect products of monoids, before going on to generalise this to semidirect products of monoidal categories in later sections.
The following definition should be familiar to anyone already familiar with semidirect products of groups.
For an overview of semidirect products of groups, see Robinson \cite{Robinson2003}.

\begin{definition}[action of a monoid]
	Let $X$ and $C$ be monoids.
	A right action (which we will simply refer to as an action) of $X$ on $C$ is a monoid homomorphism $\Gamma$ from $X$ to $[C, C]$,
	the monoid of monoid endomorphisms on $C$ with product given by function composition in `diagrammatic order'.
	Given an element $x \in X$, we will denote the corresponding endomorphism $\Gamma (x)$ by $(-)^x \colon C \rightarrow C$.
\label{definition:monoid-action} \end{definition}

Assume we have an action of a monoid $X$ on a monoid $C$.
There are several conditions that hold.
Compare these with the structure maps described after Definition (\ref{definition:action}).
Firstly, for any element $x \in X$, the function $(-)^x$ being an endomorphism on $C$ corresponds to the following two conditions.
\[ b^x c^x = (b c)^x \qquad 1_C = 1_C^x \]
Secondly, the function $\Gamma$ being a monoid homomorphism corresponds to the following two conditions.
\[ c^{x y} = (c^x)^y \qquad c^{1_X} = c \]

\begin{definition}[semidirect product of monoids]
	Given an action of a monoid $X$ on a monoid $C$, we can define a `semidirect product' monoid, $X \ltimes C$, as follows.
	The underlying set of $X \ltimes C$ is $X \times C$, and we will denote a pair $(x, c) \in X \ltimes C$ by $\langle x, c \rangle$.
	The product is defined as follows.
	\[ \langle x, b \rangle \langle y, c \rangle = \langle x y, b^y c \rangle \]
	The unit is defined as follows.
	\[ 1_{X \ltimes C} = \langle 1_X, 1_C \rangle \]
\end{definition}

We will now verify that this product is associative and unital.

\begin{lemma}
	The product just defined is associative.
\label{lemma:monoid-associative} \end{lemma}
\begin{proof}
	We must show that the two different ways of bracketing the product of three arbitrary elements $\langle x, a \rangle$, $\langle y, b \rangle$ and $\langle z, c \rangle$ are equal.
	\begin{align*}
		\langle x, a \rangle (\langle y, b \rangle \langle z, c \rangle)
		&= \langle x, a \rangle \langle y z, b^z c \rangle \\
		&= \langle x (y z), a^{y z} (b^z c) \rangle \\
		&= \langle x (y z), (a^y)^z (b^z c) \rangle \\
		&= \langle (x y) z, ((a^y)^z b^z) c \rangle \\
		&= \langle (x y) z, (a^y b)^z c \rangle \\
		&= \langle x y, a^y b \rangle \langle z, c \rangle \\
		&= (\langle x, a \rangle \langle y, b \rangle) \langle z, c \rangle
	\end{align*}
	Compare this with the definition of the associator in Definition (\ref{definition:semidirect-product}).
\end{proof}

\begin{lemma}
	The product just defined is unital.
\label{lemma:monoid-unital} \end{lemma}
\begin{proof}
	We must show that the product of an arbitrary element $\langle x, c \rangle$ with the unit element $1_{X \ltimes C}$, on either side, is equal to $\langle x, c \rangle$.
	\[ 1_{X \ltimes C} \langle x, c \rangle = \langle 1_X, 1_C \rangle \langle x, c \rangle = \langle 1_X x, (1_C)^x c \rangle = \langle 1_X x, 1_C c \rangle = \langle x, c \rangle \]
	\[ \langle x, c \rangle 1_{X \ltimes C} = \langle x, c \rangle \langle 1_X, 1_C \rangle = \langle x 1_X, c^{1_X} 1_C \rangle = \langle x 1_X, c 1_C \rangle = \langle x, c \rangle \]
	Compare this with the definitions of the unitors in Definition (\ref{definition:semidirect-product}).
\end{proof}

Throughout this paper, we will need to be familiar with the definitions of monoidal categories and the various forms of functor between them, as well as their weakened versions, skew monoidal categories.
For an overview of monoidal categories, see Leinster \cite{Leinster2003}, \S 1.2.
For an overview of skew monoidal categories, see Szlach\'anyi \cite{Szlachányi2012}.

We will use the term `skew monoidal category' to mean what is referred to by Szlach\'anyi as a `right-monoidal category'.
We will use the term `lax monoidal functor' to mean either what is referred to by Leinster as a `lax monoidal functor' or what is referred to by Szlach\'anyi as a `right-monoidal functor'.
We will use the term `oplax monoidal functor' to mean either what is referred to by Leinster as a `colax monoidal functor' or what is referred to by Szlach\'anyi as a `right-opmonoidal functor'.

In monoidal and skew monoidal categories, we will usually denote the binary tensor product by $\tensor$, the monoidal unit by $\munit$, and the associator and unitors, whether or not they are invertible, as follows.
\[ \alpha_{A, B, C} \colon A \tensor (B \tensor C) \rightarrow (A \tensor B) \tensor C \]
\[ \lambda_A \colon A \rightarrow \munit \tensor A \qquad \rho_A \colon A \tensor \munit \rightarrow A \]

\section{Actions} \label{section:actions}

In this section, we will explain how to categorify the notion of actions of monoids to actions of skew monoidal categories.
We will then spend some time going through, in detail, all of the data making up such an action.
We will assume familiarity with the concepts of strong, lax and oplax monoidal functors, and monoidal natural transformations.

As is often the case when categorifying, there are some choices to be made as to the direction certain morphisms should take and whether or not they should be invertible.
We will focus on one such choice, which happens to be convenient for our purposes, and call it simply a weak action.

\begin{definition}[weak action]
	Let $\X$ and $\C$ be skew monoidal categories.
	A `weak action' of $\X$ on $\C$ is an oplax monoidal functor $\Gamma$ from $\X$ to $[\C, \C]_\text{lax}$,
	the strict monoidal category of lax monoidal endofunctors on $\C$ and monoidal natural transformations between them with tensor product given by functor composition in `diagrammatic order'.
\label{definition:action} \end{definition}

There is quite a bit of data involved in this definition, so we will spend some time going through the structure maps involved.
Compare these with the conditions described after Definition (\ref{definition:monoid-action}).

Firstly, for every object $X \in \X$, we have a lax monoidal endofunctor on $\C$, denoted as follows.
\[ (-)^X \colon \C \rightarrow \C \]
We will denote the structure maps for this lax monoidal endofunctor as follows.
\[ \varphi^X_{B, C} \colon B^X \tensor C^X \rightarrow (B \tensor C)^X \qquad \varphi^X \colon \munit_\C \rightarrow \munit_\C^X \]
The conditions that these must satisfy are that the following three diagrams must commute.
\[ \begin{tikzpicture}[xscale=3, yscale=-2]
	\node (NW) at (-1, -1) {$A^X \tensor (B^X \tensor C^X)$};
	\node (NE) at ( 1, -1) {$(A^X \tensor B^X) \tensor C^X$};
	\node (WW) at (-1,  0) {$A^X \tensor (B \tensor C)^X$};
	\node (EE) at ( 1,  0) {$(A \tensor B)^X \tensor C^X$};
	\node (SW) at (-1,  1) {$(A \tensor (B \tensor C))^X$};
	\node (SE) at ( 1,  1) {$((A \tensor B) \tensor C)^X$};
	\draw[->] (NW) to node[auto] {$\alpha_{A^X, B^X, C^X}$} (NE);
	\draw[->] (NW) to node[auto, swap] {$A^X \tensor \varphi^X_{B, C}$} (WW);
	\draw[->] (NE) to node[auto] {$\varphi^X_{A, B} \tensor C^X$} (EE);
	\draw[->] (WW) to node[auto, swap] {$\varphi^X_{A, B \tensor C}$} (SW);
	\draw[->] (EE) to node[auto] {$\varphi^X_{A \tensor B, C}$} (SE);
	\draw[->] (SW) to node[auto, swap] {$(\alpha_{A, B, C})^X$} (SE);
\end{tikzpicture} \]
\[ \begin{tikzpicture}[xscale=2, yscale=-2]
	\node (N) at ( 0, -1) {$C^X$};
	\node (W) at (-1,  0) {$\munit_\C \tensor C^X$};
	\node (E) at ( 1,  0) {$(\munit_\C \tensor C)^X$};
	\node (S) at ( 0,  1) {$\munit_\C^X \tensor C^X$};
	\draw[->] (N) to node[auto, swap] {$\lambda_{C^X}$} (W);
	\draw[->] (N) to node[auto] {$(\lambda_C)^X$} (E);
	\draw[->] (W) to node[auto, swap] {$\varphi^X \tensor C^X$} (S);
	\draw[->] (S) to node[auto, swap] {$\varphi^X_{\munit_\C, C}$} (E);
\end{tikzpicture}
\quad
\begin{tikzpicture}[xscale=2, yscale=-2]
	\node (N) at ( 0, -1) {$C^X \tensor \munit_\C^X$};
	\node (W) at (-1,  0) {$C^X \tensor \munit_\C$};
	\node (E) at ( 1,  0) {$(C \tensor \munit_\C)^X$};
	\node (S) at ( 0,  1) {$C^X$};
	\draw[->] (W) to node[auto] {$C^X \tensor \varphi^X$} (N);
	\draw[->] (N) to node[auto] {$\varphi^X_{C, \munit_\C}$} (E);
	\draw[->] (W) to node[auto, swap] {$\rho_{C^X}$} (S);
	\draw[->] (E) to node[auto] {$(\rho_C)^X$} (S);
\end{tikzpicture} \]

Secondly, for every morphism $f \colon X \rightarrow Y$ in $\X$, we have a monoidal natural transformation $(-)^f \colon (-)^X \rightarrow (-)^Y$.
The conditions that this natural transformation must satisfy are that the following two diagrams must commute.
\[ \begin{tikzpicture}[xscale=2, yscale=-1]
	\node (NW) at (-1, -1) {$B^X \tensor C^X$};
	\node (NE) at ( 1, -1) {$B^Y \tensor C^Y$};
	\node (SW) at (-1,  1) {$(B \tensor C)^X$};
	\node (SE) at ( 1,  1) {$(B \tensor C)^Y$};
	\draw[->] (NW) to node[auto] {$B^f \tensor C^f$} (NE);
	\draw[->] (NW) to node[auto, swap] {$\varphi^X_{B, C}$} (SW);
	\draw[->] (NE) to node[auto] {$\varphi^Y_{B, C}$} (SE);
	\draw[->] (SW) to node[auto, swap] {$(B \tensor C)^f$} (SE);
\end{tikzpicture}
\quad
\begin{tikzpicture}[xscale=1.5, yscale=-1]
	\node (NN) at ( 0, -1) {$\munit_\C$};
	\node (SW) at (-1,  1) {$\munit_\C^X$};
	\node (SE) at ( 1,  1) {$\munit_\C^Y$};
	\draw[->] (NN) to node[auto, swap] {$\varphi^X$} (SW);
	\draw[->] (NN) to node[auto] {$\varphi^Y$} (SE);
	\draw[->] (SW) to node[auto, swap] {$\munit_\C^f$} (SE);
\end{tikzpicture} \]

In addition to this, the functor $\Gamma$ itself is oplax monoidal.
We will denote the structure maps for $\Gamma$ as follows.
\[ \psi^{X, Y} \colon (-)^{X \tensor Y} \Rightarrow ((-)^X)^Y \qquad \psi \colon (-)^{\munit_\X} \Rightarrow (-) \]
The conditions that these must satisfy are that the following three diagrams must commute.
Note that the monoidal category $[\C, \C]$, which is the target of $\Gamma$, is strict, so some of the edges in these diagrams are identities.
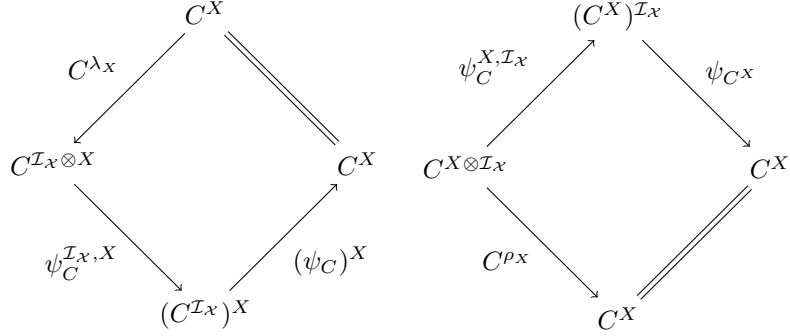
\[ \begin{tikzpicture}[xscale=3, yscale=-2]
	\node (NW) at (-1, -1) {$C^{X \tensor (Y \tensor Z)}$};
	\node (NE) at ( 1, -1) {$C^{(X \tensor Y) \tensor Z}$};
	\node (WW) at (-1,  0) {$(C^X)^{Y \tensor Z}$};
	\node (EE) at ( 1,  0) {$(C^{X \tensor Y})^Z$};
	\node (SW) at (-1,  1) {$((C^X)^Y)^Z$};
	\node (SE) at ( 1,  1) {$((C^X)^Y)^Z$};
	\draw[->] (NW) to node[auto] {$C^{\alpha_{X, Y, Z}}$} (NE);
	\draw[->] (NW) to node[auto, swap] {$\psi_C^{X, Y \tensor Z}$} (WW);
	\draw[->] (NE) to node[auto] {$\psi_C^{X \tensor Y, Z}$} (EE);
	\draw[->] (WW) to node[auto, swap] {$\psi_{C^X}^{Y, Z}$} (SW);
	\draw[->] (EE) to node[auto] {$(\psi_C^{X, Y})^Z$} (SE);
	\draw[equal] (SW) to (SE);
\end{tikzpicture} \]
\[ \begin{tikzpicture}[xscale=2, yscale=-2]
	\node (N) at ( 0, -1) {$C^X$};
	\node (W) at (-1,  0) {$C^{\munit_\X \tensor X}$};
	\node (E) at ( 1,  0) {$C^X$};
	\node (S) at ( 0,  1) {$(C^{\munit_\X})^X$};
	\draw[->] (N) to node[auto, swap] {$C^{\lambda_X}$} (W);
	\draw[equal] (N) to (E);
	\draw[->] (W) to node[auto, swap] {$\psi_C^{\munit_\X, X}$} (S);
	\draw[->] (S) to node[auto, swap] {$(\psi_C)^X$} (E);
\end{tikzpicture}
\quad
\begin{tikzpicture}[xscale=2, yscale=-2]
	\node (N) at ( 0, -1) {$(C^X)^{\munit_\X}$};
	\node (W) at (-1,  0) {$C^{X \tensor \munit_\X}$};
	\node (E) at ( 1,  0) {$C^X$};
	\node (S) at ( 0,  1) {$C^X$};
	\draw[->] (W) to node[auto] {$\psi_C^{X, \munit_\X}$} (N);
	\draw[->] (N) to node[auto] {$\psi_{C^X}$} (E);
	\draw[->] (W) to node[auto, swap] {$C^{\rho_X}$} (S);
	\draw[equal] (E) to (S);
\end{tikzpicture} \]

Finally, the components of the structure maps for $\Gamma$ are morphisms in $[\C, \C]$, which means that they are monoidal natural transformations.
The structure map
\[ \psi^{X, Y} \colon (-)^{X \tensor Y} \Rightarrow ((-)^X)^Y \]
being a monoidal natural transformation corresponds to the following two diagrams commuting.
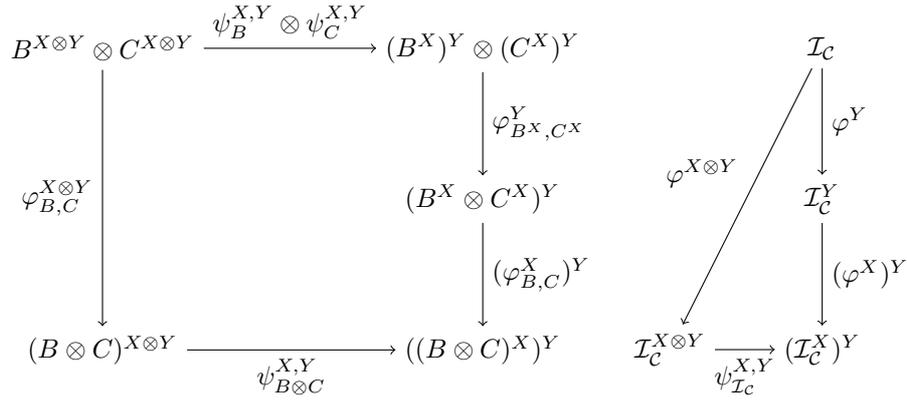
\[ \begin{tikzpicture}[xscale=2.5, yscale=-2]
	\node (NW) at (-1, -1) {$B^{X \tensor Y} \tensor C^{X \tensor Y}$};
	\node (NE) at ( 1, -1) {$(B^X)^Y \tensor (C^X)^Y$};
	\node (EE) at ( 1,  0) {$(B^X \tensor C^X)^Y$};
	\node (SW) at (-1,  1) {$(B \tensor C)^{X \tensor Y}$};
	\node (SE) at ( 1,  1) {$((B \tensor C)^X)^Y$};
	\draw[->] (NW) to node[auto] {$\psi_B^{X, Y} \tensor \psi_C^{X, Y}$} (NE);
	\draw[->] (NW) to node[auto, swap] {$\varphi^{X \tensor Y}_{B, C}$} (SW);
	\draw[->] (NE) to node[auto] {$\varphi^Y_{B^X, C^X}$} (EE);
	\draw[->] (EE) to node[auto] {$(\varphi^X_{B, C})^Y$} (SE);
	\draw[->] (SW) to node[auto, swap] {$\psi_{B \tensor C}^{X, Y}$} (SE);
\end{tikzpicture}
\quad
\begin{tikzpicture}[xscale=1, yscale=-2]
	\node (NN) at ( 1, -1) {$\munit_\C$};
	\node (EE) at ( 1,  0) {$\munit_\C^Y$};
	\node (SW) at (-1,  1) {$\munit_\C^{X \tensor Y}$};
	\node (SE) at ( 1,  1) {$(\munit_\C^X)^Y$};
	\draw[->] (NN) to node[auto] {$\varphi^Y$} (EE);
	\draw[->] (NN) to node[auto, swap] {$\varphi^{X \tensor Y}$} (SW);
	\draw[->] (EE) to node[auto] {$(\varphi^X)^Y$} (SE);
	\draw[->] (SW) to node[auto, swap] {$\psi_{\munit_\C}^{X, Y}$} (SE);
\end{tikzpicture} \]
The structure map
\[ \psi \colon (-)^{\munit_\X} \Rightarrow (-) \]
being a monoidal natural transformation corresponds to the following two diagrams commuting.
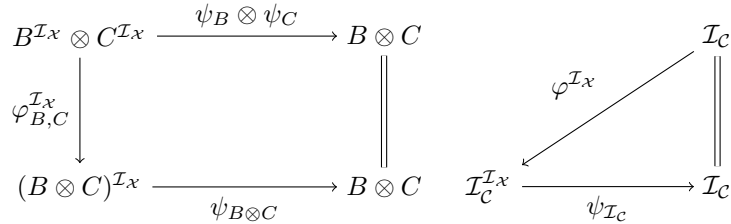
\[ \begin{tikzpicture}[xscale=2, yscale=-1]
	\node (NW) at (-1, -1) {$B^{\munit_\X} \tensor C^{\munit_\X}$};
	\node (NE) at ( 1, -1) {$B \tensor C$};
	\node (SW) at (-1,  1) {$(B \tensor C)^{\munit_\X}$};
	\node (SE) at ( 1,  1) {$B \tensor C$};
	\draw[->] (NW) to node[auto] {$\psi_B \tensor \psi_C$} (NE);
	\draw[->] (NW) to node[auto, swap] {$\varphi^{\munit_\X}_{B, C}$} (SW);
	\draw[equal] (NE) to (SE);
	\draw[->] (SW) to node[auto, swap] {$\psi_{B \tensor C}$} (SE);
\end{tikzpicture}
\quad
\begin{tikzpicture}[xscale=1.5, yscale=-1]
	\node (NN) at ( 1, -1) {$\munit_\C$};
	\node (SW) at (-1,  1) {$\munit_\C^{\munit_\X}$};
	\node (SE) at ( 1,  1) {$\munit_\C$};
	\draw[->] (NN) to node[auto, swap] {$\varphi^{\munit_\X}$} (SW);
	\draw[equal] (NN) to (SE);
	\draw[->] (SW) to node[auto, swap] {$\psi_{\munit_\C}$} (SE);
\end{tikzpicture} \]

\section{Semidirect Products} \label{section:semidirect-products}

In this section, we will explain how to categorify the notion of semidirect products of monoids to semidirect products of skew monoidal categories.

\begin{definition}[semidirect product]
	Given a weak action of skew monoidal category $\X$ on a skew monoidal category $\C$, we can define a semidirect product skew monoidal category, $\X \ltimes \C$.
	The underlying category of $\X \ltimes \C$ is $\X \times \C$, and we will denote an object $(X, C) \in \X \ltimes \C$ by $\langle X, C \rangle$.
	The tensor product is defined as follows.
	\[ \langle X, B \rangle \tensor \langle Y, C \rangle = \langle X \tensor Y, B^Y \tensor C \rangle \]
	The monoidal unit is defined as follows.
	\[ \munit_{\X \ltimes \C} = \langle \munit_\X, \munit_\C \rangle \]
	In order to define the associator and unitors, it suffices to define their images under the projection functors $\pi_\X$ and $\pi_\C$.
	\[ \X \xleftarrow{\pi_\X} \X \ltimes \C \xrightarrow{\pi_\C} \C \]
	The associator, $\alpha$, is defined as follows.
	The component
	\[ \alpha_{\langle X, A \rangle, \langle Y, B \rangle, \langle Z, C \rangle} \colon \langle X, A \rangle \tensor \langle \langle Y, B \rangle \tensor \langle Z, C \rangle \rangle \rightarrow \langle \langle X, A \rangle \tensor \langle Y, B \rangle \rangle \tensor \langle Z, C \rangle \]
	is the morphism whose images under $\pi_\X$ and $\pi_\C$ are the following pair of morphisms, respectively.
	\[ X \tensor (Y \tensor Z) \xrightarrow{\alpha_{X, Y, Z}} (X \tensor Y) \tensor Z \]
	\begin{align*}
		A^{Y \tensor Z} \tensor (B^Z \tensor C)
		&\xrightarrow{\psi_A^{Y, Z} \tensor (B^Z \tensor C)} (A^Y)^Z \tensor (B^Z \tensor C) \\
		&\xrightarrow{\alpha_{(A^Y)^Z, B^Z, C}} ((A^Y)^Z \tensor B^Z) \tensor C \\
		&\xrightarrow{\varphi_{A^Y, B}^Z \tensor C} (A^Y \tensor B)^Z \tensor C
	\end{align*}
	Compare this with the proof of Lemma (\ref{lemma:monoid-associative}).
	The left unitor, $\lambda$, is defined as follows.
	The component
	\[ \lambda_{\langle X, C \rangle} \colon \langle X, C \rangle \rightarrow \munit_{\X \ltimes \C} \tensor \langle X, C \rangle \]
	is the morphism whose images under $\pi_\X$ and $\pi_\C$ are the following pair of morphisms, respectively.
	\[ X \xrightarrow{\lambda_X} \munit_\X \tensor X \]
	\[ C \xrightarrow{\lambda_C} \munit_\C \tensor C \xrightarrow{\varphi^X \tensor C} \munit_\C^X \tensor C \]
	Compare this with the proof of Lemma (\ref{lemma:monoid-unital}).
	The right unitor, $\rho$, is defined as follows.
	The component
	\[ \rho_{\langle X, C \rangle} \colon \langle X, C \rangle \tensor \munit_{\X \ltimes \C} \rightarrow \langle X, C \rangle \]
	is the morphism whose images under $\pi_\X$ and $\pi_\C$ are the following pair of morphisms, respectively.
	\[ X \tensor \munit_\X \xrightarrow{\rho_X} X \]
	\[ C^{\munit_\X} \tensor \munit_\C \xrightarrow{\psi_C \tensor \munit_\X} C \tensor \munit_\X \xrightarrow{\rho_C} C \]
	Compare this with the proof of Lemma (\ref{lemma:monoid-unital}).
\label{definition:semidirect-product} \end{definition}

In order to show that $\X \ltimes \C$ is a skew monoidal category, we must show that the pentagon identity, the three triangle identities, and the unitor identity hold.
However, in order to show that a diagram commutes in $\X \ltimes \C$, it suffices to show that its images under the projection functors $\pi_\X$ and $\pi_\C$ commute.
And, since the images of the associator and unitors under the projection functor $\pi_\X$ are just the associator and unitors in $\X$, which is a skew monoidal category, the images of the pentagon diagram, the three triangle diagrams and the unitor diagram under the projection functor $\pi_\X$ do commute.
Hence, we only need to show that the images of the pentagon diagram, the three triangle diagrams and the unitor diagram under the projection functor $\pi_\C$ commute.

\begin{lemma}
	The pentagon identity holds.
\end{lemma}
\begin{proof}
	We will consider the pentagon identity as it applies to the four objects $\langle W, A \rangle$, $\langle X, B \rangle$, $\langle Y, C \rangle$ and $\langle Z, D \rangle$.
	Throughout this proof, we will denote tensor products by juxtaposition, for notational convenience.
	In this case, the five different bracketings which form the vertices of the pentagon are as follows.
	\[ \langle W, A \rangle (\langle X, B \rangle (\langle Y, C \rangle \langle Z, D \rangle)) = \langle W (X (Y Z)), A^{X (Y Z)} (B^{Y Z} (C^Z D)) \rangle \]
	\[ \langle W, A \rangle ((\langle X, B \rangle \langle Y, C \rangle) \langle Z, D \rangle) = \langle W ((X Y) Z), A^{(X Y) Z} ((B^Y C)^Z D) \rangle \]
	\[ (\langle W, A \rangle (\langle X, B \rangle \langle Y, C \rangle)) \langle Z, D \rangle = \langle (W (X Y)) Z, ((A^{X Y}) (B^Y C))^Z D \rangle \]
	\[ ((\langle W, A \rangle \langle X, B \rangle) \langle Y, C \rangle) \langle Z, D \rangle = \langle ((W X) Y) Z, (((A^X B)^Y) C)^Z D \rangle \]
	\[ (\langle W, A \rangle \langle X, B \rangle) (\langle Y, C \rangle \langle Z, D \rangle) = \langle (W X) (Y Z), (A^X B)^{Y Z} (C^Z D) \rangle \]
	The image under $\pi_\C$ of the pentagon diagram is shown, without the arrows labelled, in Figure (\ref{figure:pentagon}).
	Hopefully, the contents of the arrows are clear from context.
	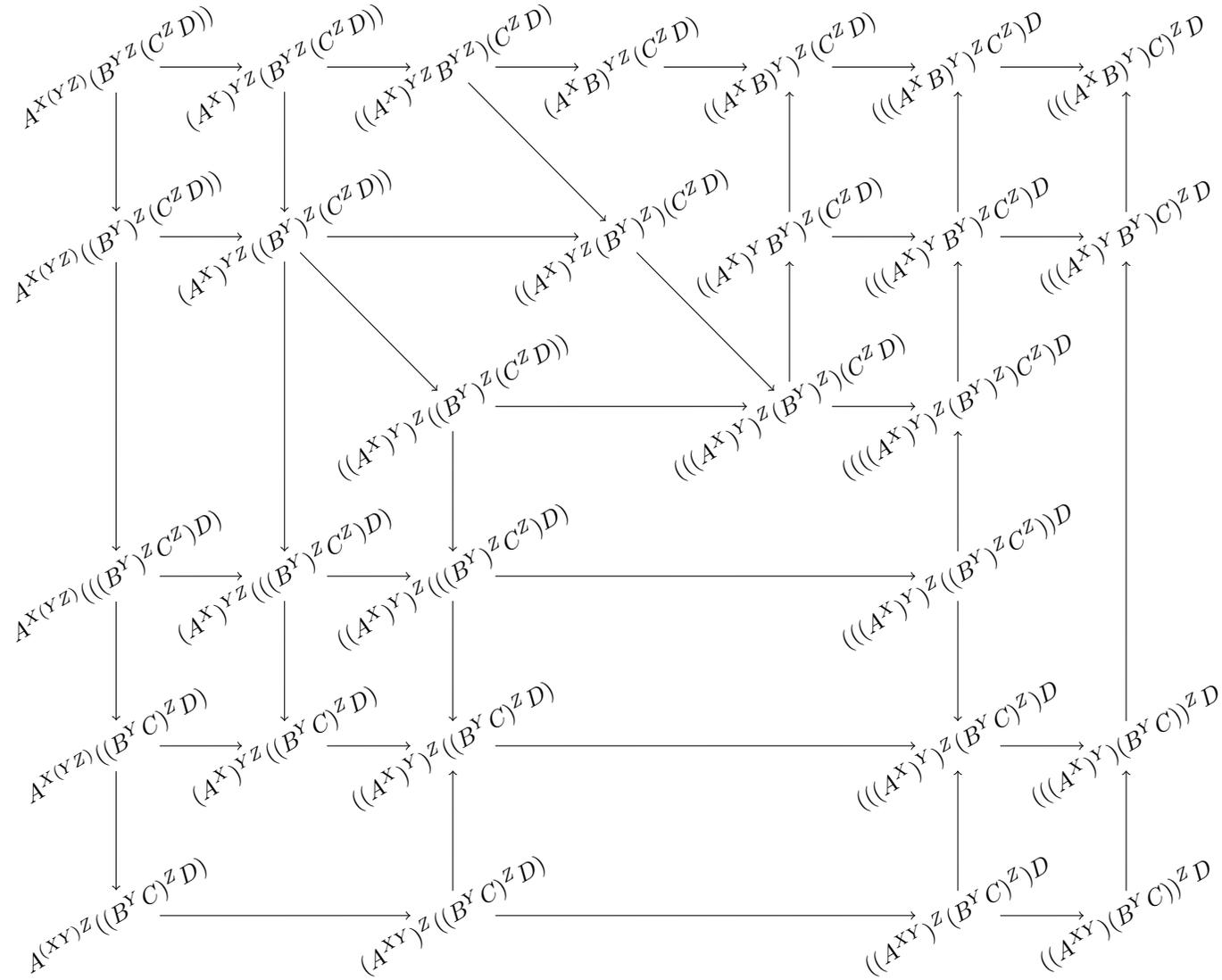
\begin{sidewaysfigure}
		\[ \begin{tikzpicture}[xscale=2.5, yscale=-2.5]
			\node[compactobject] (00) at (0, 0) {$A^{X (Y Z)} (B^{Y Z} (C^Z D))$};
			\node[compactobject] (10) at (1, 0) {$(A^X)^{Y Z} (B^{Y Z} (C^Z D))$};
			\node[compactobject] (20) at (2, 0) {$((A^X)^{Y Z} B^{Y Z}) (C^Z D)$};
			\node[compactobject] (30) at (3, 0) {$(A^X B)^{Y Z} (C^Z D)$};
			\node[compactobject] (40) at (4, 0) {$((A^X B)^Y)^Z (C^Z D)$};
			\node[compactobject] (50) at (5, 0) {$(((A^X B)^Y)^Z C^Z) D$};
			\node[compactobject] (60) at (6, 0) {$(((A^X B)^Y) C)^Z D$};
			\node[compactobject] (01) at (0, 1) {$A^{X (Y Z)} ((B^Y)^Z (C^Z D))$};
			\node[compactobject] (11) at (1, 1) {$(A^X)^{Y Z} ((B^Y)^Z (C^Z D))$};
			\node[compactobject] (31) at (3, 1) {$((A^X)^{Y Z} (B^Y)^Z) (C^Z D)$};
			\node[compactobject] (41) at (4, 1) {$((A^X)^Y B^Y)^Z (C^Z D)$};
			\node[compactobject] (51) at (5, 1) {$(((A^X)^Y B^Y)^Z C^Z) D$};
			\node[compactobject] (61) at (6, 1) {$(((A^X)^Y B^Y) C)^Z D$};
			\node[compactobject] (22) at (2, 2) {$((A^X)^Y)^Z ((B^Y)^Z (C^Z D))$};
			\node[compactobject] (42) at (4, 2) {$(((A^X)^Y)^Z (B^Y)^Z) (C^Z D)$};
			\node[compactobject] (52) at (5, 2) {$((((A^X)^Y)^Z (B^Y)^Z) C^Z) D$};
			\node[compactobject] (03) at (0, 3) {$A^{X (Y Z)} (((B^Y)^Z C^Z) D)$};
			\node[compactobject] (13) at (1, 3) {$(A^X)^{Y Z} (((B^Y)^Z C^Z) D)$};
			\node[compactobject] (23) at (2, 3) {$((A^X)^Y)^Z (((B^Y)^Z C^Z) D)$};
			\node[compactobject] (53) at (5, 3) {$(((A^X)^Y)^Z ((B^Y)^Z C^Z)) D$};
			\node[compactobject] (04) at (0, 4) {$A^{X (Y Z)} ((B^Y C)^Z D)$};
			\node[compactobject] (14) at (1, 4) {$(A^X)^{Y Z} ((B^Y C)^Z D)$};
			\node[compactobject] (24) at (2, 4) {$((A^X)^Y)^Z ((B^Y C)^Z D)$};
			\node[compactobject] (54) at (5, 4) {$(((A^X)^Y)^Z (B^Y C)^Z) D$};
			\node[compactobject] (64) at (6, 4) {$(((A^X)^Y) (B^Y C))^Z D$};
			\node[compactobject] (05) at (0, 5) {$A^{(X Y) Z} ((B^Y C)^Z D)$};
			\node[compactobject] (25) at (2, 5) {$(A^{X Y})^Z ((B^Y C)^Z D)$};
			\node[compactobject] (55) at (5, 5) {$((A^{X Y})^Z (B^Y C)^Z) D$};
			\node[compactobject] (65) at (6, 5) {$((A^{X Y}) (B^Y C))^Z D$};
			\draw[->] (00) to (01); \draw[->] (01) to (03); \draw[->] (03) to (04); \draw[->] (04) to (05);
			\draw[->] (10) to (11); \draw[->] (11) to (13); \draw[->] (13) to (14);
			\draw[->] (22) to (23); \draw[->] (23) to (24); \draw[<-] (24) to (25);
			\draw[<-] (40) to (41); \draw[<-] (41) to (42);
			\draw[<-] (50) to (51); \draw[<-] (51) to (52); \draw[<-] (52) to (53); \draw[->] (53) to (54); \draw[<-] (54) to (55);
			\draw[<-] (60) to (61); \draw[<-] (61) to (64); \draw[<-] (64) to (65);
			\draw[->] (00) to (10); \draw[->] (10) to (20); \draw[->] (20) to (30); \draw[->] (30) to (40); \draw[->] (40) to (50); \draw[->] (50) to (60);
			\draw[->] (01) to (11); \draw[->] (11) to (31); \draw[->] (41) to (51); \draw[->] (51) to (61);
			\draw[->] (22) to (42); \draw[->] (42) to (52);
			\draw[->] (03) to (13); \draw[->] (13) to (23); \draw[->] (23) to (53);
			\draw[->] (04) to (14); \draw[->] (14) to (24); \draw[->] (24) to (54); \draw[->] (54) to (64);
			\draw[->] (05) to (25); \draw[->] (25) to (55); \draw[->] (55) to (65);
			\draw[->] (11) to (22);
			\draw[->] (20) to (31); \draw[->] (31) to (42);
		\end{tikzpicture} \]
		\caption{Proof of the pentagon identity}
		\label{figure:pentagon}
	\end{sidewaysfigure}
\end{proof}

\begin{lemma}
	The triangle identities hold.
\end{lemma}
\begin{proof}
	We will consider the triangle identities as they apply to the two objects $\langle X, B \rangle$ and $\langle Y, C \rangle$.
	Throughout this proof, we will denote tensor products by juxtaposition, for notational convenience.

	In the case of the first triangle identity, the three different bracketings which form the vertices of the triangle are as follows.
	\[ \munit_{\X \ltimes \C} (\langle X, B \rangle \langle Y, C \rangle) = \langle \munit_\X (X Y), \munit_\C^{X Y} (B^Y C) \rangle \]
	\[ (\munit_{\X \ltimes \C} \langle X, B \rangle) \langle Y, C \rangle = \langle (\munit_\X X) Y, (\munit_\C^X B)^Y C \rangle \]
	\[ \langle X, B \rangle \langle Y, C \rangle = \langle X Y, B^Y C \rangle \]
	The image under $\pi_\C$ of the first triangle diagram is shown in Figure (\ref{figure:triangle-l}).
	\begin{sidewaysfigure}
		\[ \begin{tikzpicture}[xscale=5, yscale=-3]
			\node[compactobject] (00) at (0, 0) {$\munit_\C^{X Y} (B^Y C)$};
			\node[compactobject] (10) at (1, 0) {$(\munit_\C^X)^Y (B^Y C)$};
			\node[compactobject] (20) at (2, 0) {$((\munit_\C^X)^Y B^Y) C$};
			\node[compactobject] (30) at (3, 0) {$(\munit_\C^X B)^Y C$};
			\node[compactobject] (11) at (1, 1) {$\munit_\C^Y (B^Y C)$};
			\node[compactobject] (21) at (2, 1) {$(\munit_\C^Y B^Y) C$};
			\node[compactobject] (31) at (3, 1) {$(\munit_\C B)^Y C$};
			\node[compactobject] (12) at (1, 2) {$\munit_\C (B^Y C)$};
			\node[compactobject] (22) at (2, 2) {$(\munit_\C B^Y) C$};
			\node[compactobject] (13) at (1, 3) {$B^Y C$};
			\node[compactobject] (23) at (2, 3) {$B^Y C$};
			\draw[->] (00) to node[compactarrow] {$\psi_{\munit_\C}^{X, Y} (B^Y C)$} (10); \draw[->] (10) to node[compactarrow] {$\alpha_{(\munit_\C^X)^Y, B^Y, C}$} (20); \draw[->] (20) to node[compactarrow] {$\varphi^Y_{\munit_\C^X, B} C$} (30);
			\draw[->] (11) to node[compactarrow] {$\alpha_{\munit_\C^Y, B^Y, C}$} (21); \draw[->] (21) to node[compactarrow] {$\varphi^Y_{\munit_\C, B} C$} (31);
			\draw[->] (12) to node[compactarrow] {$\alpha_{\munit_\C, B^Y, C}$} (22);
			\draw[equal] (13) to (23);
			\draw[<-] (10) to node[compactarrow] {$(\varphi^X)^Y (B^Y C)$} (11); \draw[<-] (11) to node[compactarrow] {$\varphi^Y (B^Y C)$} (12); \draw[<-] (12) to node[compactarrow] {$\lambda_{B^Y C}$} (13);
			\draw[<-] (20) to node[compactarrow] {$((\varphi^X)^Y B^Y) C$} (21); \draw[<-] (21) to node[compactarrow] {$(\varphi^Y B^Y) C$} (22); \draw[<-] (22) to node[compactarrow] {$\lambda_{B^Y} C$} (23);
			\draw[<-] (30) to node[compactarrow] {$(\varphi^X B)^Y C$} (31);
			\draw[->] (12) to node[compactarrow] {$\varphi^{X Y} (B^Y C)$} (00);
			\draw[->] (23) to node[compactarrow] {$(\lambda_B)^Y C$} (31);
		\end{tikzpicture} \]
		\caption{Proof of the first triangle identity}
		\label{figure:triangle-l}
	\end{sidewaysfigure}
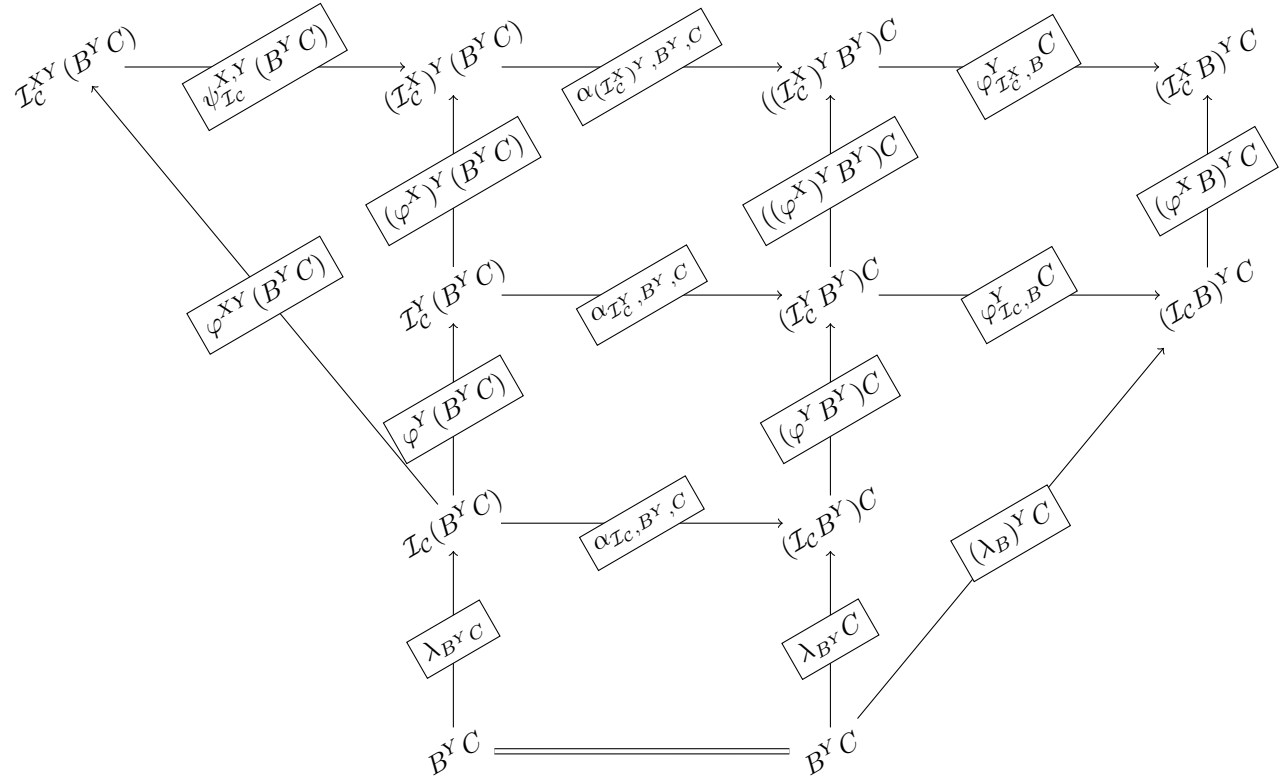

	In the case of the second triangle identity, the three different bracketings which form the vertices of the triangle are as follows.
	\[ \langle X, B \rangle (\munit_{\X \ltimes \C} \langle Y, C \rangle) = \langle X (\munit_\X Y), B^{\munit_\X Y} (\munit_\C^Y C) \rangle \]
	\[ (\langle X, B \rangle \munit_{\X \ltimes \C}) \langle Y, C \rangle = \langle (X \munit_\X) Y, (B^{\munit_\X} \munit_\C)^Y C \rangle \]
	\[ \langle X, B \rangle \langle Y, C \rangle = \langle X Y, B^Y C \rangle \]
	The image under $\pi_\C$ of the second triangle diagram is shown in Figure (\ref{figure:triangle-m}).
	\begin{sidewaysfigure}
		\[ \begin{tikzpicture}[xscale=5, yscale=-3]
			\node[compactobject] (00) at (0, 0) {$B^{\munit_\X Y} (\munit_\C^Y C)$};
			\node[compactobject] (10) at (1, 0) {$(B^{\munit_\X})^Y (\munit_\C^Y C)$};
			\node[compactobject] (20) at (2, 0) {$((B^{\munit_\X})^Y \munit_\C^Y) C$};
			\node[compactobject] (30) at (3, 0) {$(B^{\munit_\X} \munit_\C)^Y C$};
			\node[compactobject] (11) at (1, 1) {$B^Y (\munit_\C^Y C)$};
			\node[compactobject] (21) at (2, 1) {$(B^Y \munit_\C^Y) C$};
			\node[compactobject] (31) at (3, 1) {$(B \munit_\C)^Y C$};
			\node[compactobject] (12) at (1, 2) {$B^Y (\munit_\C C)$};
			\node[compactobject] (22) at (2, 2) {$(B^Y \munit_\C) C$};
			\node[compactobject] (13) at (1, 3) {$B^Y C$};
			\node[compactobject] (23) at (2, 3) {$B^Y C$};
			\draw[->] (00) to node[compactarrow] {$\psi_B^{\munit_\X, Y} (\munit_\C^Y C)$} (10); \draw[->] (10) to node[compactarrow] {$\alpha_{(B^{\munit_\X})^ Y, \munit_\C^Y, C}$} (20); \draw[->] (20) to node[compactarrow] {$\varphi^Y_{B^{\munit_\X}, \munit_\C} C$} (30);
			\draw[->] (11) to node[compactarrow] {$\alpha_{B^Y, \munit_\C^Y, C}$} (21); \draw[->] (21) to node[compactarrow] {$\varphi^Y_{B, \munit_C} C$} (31);
			\draw[->] (12) to node[compactarrow] {$\alpha_{B^Y, \munit_C, C}$} (22);
			\draw[equal] (13) to (23);
			\draw[->] (10) to node[compactarrow] {$(\psi_B)^Y (\munit_\C^Y C)$} (11); \draw[<-] (11) to node[compactarrow] {$B^Y (\varphi^Y C)$} (12); \draw[<-] (12) to node[compactarrow] {$B^Y \lambda_C$} (13);
			\draw[->] (20) to node[compactarrow] {$((\psi_B)^Y \munit_\C^Y) C$} (21); \draw[<-] (21) to node[compactarrow] {$(B^Y \varphi^Y) C$} (22); \draw[->] (22) to node[compactarrow] {$\rho_{B^Y} C$} (23);
			\draw[->] (30) to node[compactarrow] {$(\psi_B \munit_\C)^Y C$} (31);
			\draw[->] (11) to node[compactarrow] {$B^{\lambda_Y} (\munit_\C^Y C)$} (00);
			\draw[->] (31) to node[compactarrow] {$(\rho_B)^Y C$} (23);
		\end{tikzpicture} \]
		\caption{Proof of the second triangle identity}
		\label{figure:triangle-m}
	\end{sidewaysfigure}

	In the case of the third triangle identity, the three different bracketings which form the vertices of the triangle are as follows.
	\[ \langle X, B \rangle (\langle Y, C \rangle \munit_{\X \ltimes \C}) = \langle X (Y \munit_\X), A^{X \munit_\X} (B^{\munit_\X} \munit_\C) \rangle \]
	\[ (\langle X, B \rangle \langle Y, C \rangle) \munit_{\X \ltimes \C} = \langle (X Y) \munit_\X, (A^X B)^{\munit_\X} \munit_\C \rangle \]
	\[ \langle X, B \rangle \langle Y, C \rangle = \langle X Y, B^Y C \rangle \]
	The image under $\pi_\C$ of the third triangle diagram is shown in Figure (\ref{figure:triangle-r}).
	\begin{sidewaysfigure}
		\[ \begin{tikzpicture}[xscale=5, yscale=-3]
			\node[compactobject] (00) at (0, 0) {$A^{X \munit_\X} (B^{\munit_\X} \munit_\C)$};
			\node[compactobject] (10) at (1, 0) {$(A^X)^{\munit_\X} (B^{\munit_\X} \munit_\C)$};
			\node[compactobject] (20) at (2, 0) {$((A^X)^{\munit_\X} B^{\munit_\X}) \munit_\C$};
			\node[compactobject] (30) at (3, 0) {$(A^X B)^{\munit_\X} \munit_\C$};
			\node[compactobject] (01) at (0, 1) {$A^{X \munit_\X} B^{\munit_\X}$};
			\node[compactobject] (11) at (1, 1) {$(A^X)^{\munit_\X} B^{\munit_\X}$};
			\node[compactobject] (21) at (2, 1) {$(A^X)^{\munit_\X} B^{\munit_\X}$};
			\node[compactobject] (31) at (3, 1) {$(A^X B)^{\munit_\X}$};
			\node[compactobject] (02) at (0, 2) {$A^{X \munit_\X} B$};
			\node[compactobject] (12) at (1, 2) {$(A^X)^{\munit_\X} B$};
			\node[compactobject] (22) at (2, 2) {$(A^X)^{\munit_\X} B$};
			\node[compactobject] (13) at (1, 3) {$A^X B$};
			\node[compactobject] (23) at (2, 3) {$A^X B$};
			\draw[->] (00) to node[compactarrow] {$\psi_A^{X, \munit_\X} (B^\munit_\X \munit_\C)$} (10); \draw[->] (10) to node[compactarrow] {$\alpha_{(A^X)^{\munit_\X}, B^\munit_\X, \munit_\C}$} (20); \draw[->] (20) to node[compactarrow] {$\varphi^{\munit_\X}_{A^X, B} \munit_\C$} (30);
			\draw[->] (01) to node[compactarrow] {$\psi_A^{X, \munit_\X} B^{\munit_\X}$} (11); \draw[equal] (11) to (21); \draw[->] (21) to node[compactarrow] {$\varphi^{\munit_\X}_{A^X, B}$} (31);
			\draw[->] (02) to node[compactarrow] {$\psi_A^{X, \munit_\X} B$} (12); \draw[equal] (12) to (22);
			\draw[equal] (13) to (23);
			\draw[->] (00) to node[compactarrow] {$A^{X \munit_\X} \rho_{B^{\munit_\X}}$} (01); \draw[->] (01) to node[compactarrow] {$A^{X \munit_\X} \psi_B$} (02);
			\draw[->] (10) to node[compactarrow] {$(A^X)^{\munit_\X} \rho_{B^{\munit_\X}}$} (11); \draw[->] (11) to node[compactarrow] {$(A^X)^{\munit_\X} \psi_B$} (12); \draw[->] (12) to node[compactarrow] {$\psi_{A^X} B$} (13);
			\draw[->] (20) to node[compactarrow] {$\rho_{(A^X)^{\munit_\X} B^{\munit_\X}}$} (21); \draw[->] (21) to node[compactarrow] {$(A^X)^{\munit_\X} \psi_B$} (22); \draw[->] (22) to node[compactarrow] {$\psi_{A^X} B$} (23);
			\draw[->] (30) to node[compactarrow] {$\rho_{(A^X B)^{\munit_\X}}$} (31);
			\draw[->] (02) to node[compactarrow] {$A^{\rho_X} B$} (13);
			\draw[->] (31) to node[compactarrow] {$\psi_{A^X B}$} (23);
		\end{tikzpicture} \]
		\caption{Proof of the third triangle identity}
		\label{figure:triangle-r}
	\end{sidewaysfigure}
\end{proof}

\begin{lemma}
	The unitor identity holds.
\end{lemma}
\begin{proof}
	The image under $\pi_\C$ of the unitor diagram is the following.
	\[ \begin{tikzpicture}[xscale=2, yscale=-2]
		\node (NW) at (0, 0) {$\munit_\C$};
		\node (WW) at (1, 1) {$\munit_\C \tensor \munit_\C$};
		\node (SS) at (2, 2) {$\munit_\C^{\munit_\X} \tensor \munit_\C$};
		\node (EE) at (3, 1) {$\munit_\C \tensor \munit_\C$};
		\node (NE) at (4, 0) {$\munit_\C$};
		\draw[->] (NW) to node[auto, swap] {$\lambda_{\munit_\C}$} (WW);
		\draw[->] (WW) to node[auto, swap] {$\varphi^{\munit_\X} \tensor \munit_\C$} (SS);
		\draw[->] (SS) to node[auto, swap] {$\psi_{\munit_\C} \tensor \munit_\C$} (EE);
		\draw[->] (EE) to node[auto, swap] {$\rho_{\munit_\C}$} (NE);
		\draw[equal] (NW) to (NE);
		\draw[equal] (WW) to (EE);
	\end{tikzpicture} \]
\end{proof}

We will now give some examples of semidirect product skew monoidal categories.

\begin{example}
	Let $\X = \{\star\}$, the monoidal category with one object and one morphism.
	Let $\C$ be a monoidal category.
	Then a weak action of $\X$ on $\C$ endows the endofunctor $(-)^\star$ with the structure of a lax monoidal comonad on $\C$; in fact, all lax monoidal comonads are of this form.
	Given such a weak action, the resulting semidirect product $\X \ltimes \C$ is a skew monoidal structure on $\{\star\} \times \C \iso \C$, with tensor product defined as follows.
	\[ B \tensor C = B^\star \tensor C \]
	This is what is referred to by Szlach\'anyi \cite{Szlachányi2012} as the skew monoidal category `corepresented' by the lax monoidal comonad $(-)^\star$.
\end{example}

\begin{example}
	This example involves generalised metric spaces, as described by Lawvere \cite{Lawvere1973}.
	Let $\X$ be $[0, \infty]$, the category whose objects are the non-negative real numbers and positive infinity, with a unique morphism $x \rightarrow y$ if and only if $x \geq y$, considered as a cocartesian monoidal category.
	Let $\C$ be the closed symmetric monoidal category of generalised metric spaces.
	Then there is a weak action of $\X$ on $\C$, given by truncation, in which the underlying set of $M^x$ is the same as the underlying set of $M$, but with a new truncated metric, defined as follows.
	\[ M^x (m, m') = \min (M (m, m'), x) \]
	The resulting semidirect product $\X \ltimes \C$ is a skew monoidal structure on $[0, \infty] \times \C$, with tensor product defined as follows.
	\[ \langle x, M \rangle \tensor \langle y, N \rangle = \langle \min (x, y), M^y \tensor N \rangle \]
	The metric on the generalised metric space $M^y \tensor N$ is given as follows.
	\[ (M^y \tensor N) ((m, n), (m', n')) = \min (M (m, m'), y) + N (n, n') \]
\end{example}

\section{Monoidal Categories} \label{section:monoidal}

In this section, we will discuss the possibility of a semidirect product skew monoidal category being a monoidal category.

\begin{definition}[strong action]
	Let $\X$ and $\C$ be monoidal categories.
	A `strong action' of $\X$ on $\C$ is a strong monoidal functor $\Gamma$ from $\X$ to $[\C, \C]$,
	the strict monoidal category of strong monoidal endofunctors on $\C$ and monoidal natural transformations between them with tensor product given by functor composition in `diagrammatic order'.
	Equivalently, this is the same as a weak action of $\X$ on $\C$ in which all of the structure maps below are invertible.
	\[ \varphi^X_{B, C} \colon B^X \tensor C^X \rightarrow (B \tensor C)^X \qquad \varphi^X \colon \munit_\C \rightarrow \munit_\C^X \]
	\[ \psi_C^{X, Y} \colon C^{X \tensor Y} \rightarrow (C^X)^Y \qquad \psi_C \colon C^{\munit_\X} \rightarrow C \]
\end{definition}

\begin{theorem}
	Let $\X$ and $\C$ be monoidal categories.
	Let there be a strong action of $\X$ on $\C$.
	Then the semidirect product $\X \ltimes \C$ is a monoidal category.
\label{theorem:monoidal} \end{theorem}
\begin{proof}
	In order to show that a skew monoidal category is a monoidal category, it suffices to show that the coherence data is invertible.
	The coherence data in $\X \ltimes \C$ are made up from the coherence data in $\X$, the coherence data in $\C$ and the structure maps of the action.
	Since all of these are invertible, it follows that the coherence data in $\X \ltimes \C$ is invertible, and that $\X \ltimes \C$ is a monoidal category.
\end{proof}

The semidirect products of monoidal categories introduced here are a special case of the distributive laws for pseudomonads introduced by Marmolejo \cite{Marmolejo1999}.
A monoidal category can be regarded as a pseudomonad in the 1-object 3-category obtained as the delooping of the cartesian monoidal 2-category $\Cat$.
A strong action of one monoidal category on another is then precisely a distributive law between these pseudomonads which is partially trivial in a particular sense.

We will now give some examples of semidirect product monoidal categories.

\begin{example}
	Let $\C$ be a closed cartesian category.
	Let $\X = \C^\op$.
	Then there is a strong action of $\X$ on $\C$ using the internal hom, defined as follows.
	\[ C^X = [X, C] \]
	The resulting semidirect product $\X \ltimes \C$ is a monoidal structure on $\C^\op \times \C$, with tensor product defined as follows.
	\[ \langle X, B \rangle \tensor \langle Y, C \rangle = \langle X \times Y, [Y, B] \times C \rangle \]
\end{example}

\begin{example}
	Let $\C$ be a closed cartesian category with finite coproducts.
	Let $\X$ be $\C$, considered as a cocartesian monoidal category.
	Choose an object $J$ in $\C$.
	Then there is a strong action of $\X$ on $\C$ using the internal hom, defined as follows.
	\[ C^X = [[X, J], C] \]
	The resulting semidirect product $\X \ltimes \C$ is a monoidal structure on $\C \times \C$, with tensor product defined as follows.
	\[ \langle X, B \rangle \tensor \langle Y, C \rangle = \langle X + Y, [[Y, J], B] \times C \rangle \]
\end{example}

\begin{example}
	Choose a category, $\J$, and a monoidal category, $\M$.
	Let $\X$ be $[\J, \J]$, the strict monoidal category of endofunctors of $\J$ and natural transformations between them with tensor product given by functor composition in `functional order'.
	Let $\C$ be $[\J, \M]$; this category inherits a monoidal structure from that of $\M$.
	Then there is a strong action of $\X$ on $\C$ given by precomposition, defined as follows.
	\[ C^X = C \circ X \]
	The resulting semidirect product $\X \ltimes \C$ is a monoidal structure on $[\J, \J] \times [\J, \M]$, with tensor product defined as follows.
	\[ \langle F, B \rangle \tensor \langle G, C \rangle = \langle F \circ G, (B \circ G) \tensor C \rangle \]
\end{example}

\begin{example}
	As a specific case of the previous example, let $\J$ be a group, $G$, considered as a 1-object groupoid, and let $\M$ be $\Vect$, the category of vector spaces and linear maps.
	Then $\X$ is a category whose objects are endomorphisms of $G$ and $\C$ is $\Rep (G)$, the category of representations of $G$.
	The resulting semidirect product $\X \ltimes \C$ is a monoidal structure on $[G, G] \times \Rep (G)$, with tensor product defined as follows.
	\[ \langle f, U \rangle \tensor \langle g, V \rangle = \langle f \circ g, g^\star (U) \tensor V \rangle \]
\end{example}

\begin{example}
	This example involves generalised metric spaces, as described by Lawvere \cite{Lawvere1973}.
	Let $\X$ be the strict closed symmetric monoidal category $\{F \rightarrow T\}$ of truth values, with tensor product given by logical conjuction and internal hom given by logical implication.
	Let $\C$ be the symmetric monoidal category of generalised metric spaces.
	Then there is a strong action of $\X$ on $\C$, in which the underlying set of $M^x$ is the same as the underlying set of $M$, but with a new metric, defined as follows.
	\begin{align*}
		M^T (m, m') &= M (m, m') \\
		M^F (m, m') &= \begin{cases} 0 & \text{if}~ M (m, m') = 0 \\ \infty & \text{if}~ M (m, m') > 0 \end{cases}
	\end{align*}
	The resulting semidirect product $\X \ltimes \C$ is a monoidal structure on $\X \times \C$, with tensor product defined as follows.
	\[ \langle x, M \rangle \tensor \langle y, N \rangle = \langle x \tensor y, M^y \tensor N \rangle \]
	The metric on the generalised metric space $M^y \tensor N$ with underlying set $M \times N$ is given as follows.
	\begin{align*}
		(M^T \tensor N) ((m, n), (m', n')) &= M (m, m') + N (n, n') \\
		(M^F \tensor N) ((m, n), (m', n')) &= \begin{cases} N (n, n') & \text{if}~ M (m, m') = 0 \\ \infty & \text{if}~ M (m, m') > 0 \end{cases}
	\end{align*}
\label{example:counterexample-closed} \end{example}

\section{Autonomous Monoidal Categories} \label{section:autonomous}

In this section, we will discuss the possibility of a semidirect product monoidal category being left autonomous.
We will give a sufficient condition for an object in a semidirect product monoidal category to have a left dual,
and therefore a sufficient condition for a semidirect product monoidal category to be left autonomous.
Analogous statements hold for right duals and right autonomous monoidal categories.

Throughout this section, we will need to be familiar with the notion of duality in monoidal categories; for an overview, see Joyal--Street \cite{JoyalStreet1993}, \S 7.
We will denote the left dual of an object $A$ by $\duall{A}$ and the evaluation and coevaluation morphisms as follows.
\[ \varepsilon^A \colon \duall{A} \tensor A \rightarrow \munit \qquad \eta^A \colon \munit \rightarrow A \tensor \duall{A} \]

Let $\X$ and $\C$ be monoidal categories.
Let there be a strong action of $\X$ on $\C$.
As before, there is a semidirect product monoidal category, $\X \ltimes \C$, with tensor product defined as follows.
\[ \langle X, B \rangle \tensor \langle Y, C \rangle = \langle X \tensor Y, B^Y \tensor C \rangle \]

Strong monoidal functors preserve duals.
In particular, the strong monoidal functor $\Gamma \colon \X \rightarrow [\C, \C]$ preserves duals; this means that, for each object $X$ in $\X$ with a left dual $\duall{X}$, we have an adjunction $(-)^X \dashv (-)^\duall{X}$.
See \cite{JoyalStreet1993} for details.

We will now give a sufficient condition for an object in a semidirect product monoidal category to have a left dual,
\begin{lemma}
	If $X$ and $A$ have left duals $\duall{X}$ and $\duall{A}$, then $\langle X, A \rangle$ has a left dual, defined as follows.
	\[ \duall{\langle X, A \rangle} = \langle \duall{X}, (\duall{A})^\duall{X} \rangle \]
\end{lemma}
\begin{proof}
	We have the following natural isomorphism of hom sets.
	\begin{align*}
		(\X \ltimes \C) (\langle X, A \rangle \tensor \langle Y, B \rangle, \langle Z, C \rangle)
		&= (\X \ltimes \C) (\langle X \tensor Y, A^Y \tensor B \rangle, \langle Z, C \rangle) \\
		&= \X (X \tensor Y, Z) \times \C (A^Y \tensor B, C) \\
		&\iso \X (X, Z \tensor \duall{Y}) \times \C (A^Y, C \tensor \duall{B}) \\
		&\iso \X (X, Z \tensor \duall{Y}) \times \C (A, (C \tensor \duall{B})^\duall{Y}) \\
		&\iso \X (X, Z \tensor \duall{Y}) \times \C (A, C^\duall{Y} \tensor (\duall{B})^\duall{Y}) \\
		&= (\X \ltimes \C) (\langle X, A \rangle, \langle Z \tensor \duall{Y}, C^\duall{Y} \tensor (\duall{B})^\duall{Y} \rangle) \\
		&= (\X \ltimes \C) (\langle X, A \rangle, \langle Z, C \rangle \tensor \langle \duall{Y}, (\duall{B})^\duall{Y} \rangle) \\
		&= (\X \ltimes \C) (\langle X, A \rangle, \langle Z, C \rangle \tensor \duall{\langle Y, B \rangle})
	\end{align*}
\end{proof}

This leads to the following sufficient condition for a semidirect product monoidal category to be left autonomous.
\begin{theorem}
	Let $\X$ and $\C$ be left autonomous monoidal categories.
	Let there be a strong action of $\X$ on $\C$.
	Then the semidirect product $\X \ltimes \C$ is a left autonomous monoidal category.
\label{theorem:left-autonomous} \end{theorem}

We will now give an example of a left autonomous semidirect product monoidal category.

\begin{example}
	Choose a field $\mathbb{K}$ and an integer $k$.
	Let $\X$ be $\mathbb{K}^\star$, the multiplicative group of non-zero elements of $\mathbb{K}$, interpreted as a discrete symmetric monoidal category.
	Let $\C$ be the category of finite dimensional vector spaces over $\mathbb{K}$.
	Then there is an action of $\X$ on $\C$, in which each strong monoidal endofunctor $(-)^x$ is the identity functor and all the detail is contained in the structure morphisms, defined as follows.
	\[ \varphi^x_{B, C} = x^k \cdot \id_{B \tensor C} \qquad \varphi^x = x^{-k} \cdot \id_{\munit_\C} \]
	\[ \psi_C^{x, y} = \id_C \qquad \psi_C = \id_C \]
	The resulting semidirect product $\X \ltimes \C$ is a left autonomous monoidal structure on $\mathbb{K}^\star \times \Vect_\mathbb{K}$, with tensor product and left duals defined as follows.
	\[ \langle x, B \rangle \tensor \langle y, C \rangle = \langle x \cdot y, B \tensor C \rangle \]
	\[ \duall{\langle x, C \rangle} = \langle x^\inv, \duall{C} \rangle \]
	These agree with the tensor product and left duals in $\mathbb{K}^\star \times \Vect_\mathbb{K}$; however, some of the associator, unitors, evaluation and coevaluation morphisms have been deformed, as follows.
	\[ \pi_\C (\alpha_{\langle x, A \rangle, \langle y, B \rangle, \langle z, C \rangle}) = z^k \cdot \alpha_{A, B, C} \]
	\[ \pi_\C (\lambda_{\langle x, C \rangle}) = x^{-k} \cdot \lambda_C \qquad \pi_\C (\rho_{\langle x, C \rangle}) = \rho_C \]
	\[ \pi_\C (\varepsilon^{\langle x, C \rangle}) = x^{-2k} \cdot \varepsilon^C \qquad \pi_\C (\eta^{\langle x, C \rangle}) = \eta^C \]
\end{example}

\section{Right Closed Monoidal Categories} \label{section:right-closed}

In this section, we will discuss the possibility of a semidirect product monoidal category being right closed.
We will show that the analogue of Theorem (\ref{theorem:left-autonomous}) for right closed monoidal categories is not true.
We will then give a sufficient condition for a semidirect product monoidal category to be right closed.

Throughout this section, we will use the term `right closed monoidal category' to mean a monoidal category $\C$ in which tensoring on the right has a right adjoint.
We will denote the internal hom as follows.
\[ [-, -] \colon \C^\op \times \C \rightarrow \C \]
Using this notation, the hom-tensor adjunction is the following natural isomorphism of hom sets.
\[ \C (A \tensor B, C) \iso \C (A, [B, C]) \]

We might hope for the following analogue of Theorem (\ref{theorem:left-autonomous}) for right closed monoidal categories to hold.
\begin{conjecture}
	Let $\X$ and $\C$ be right closed monoidal categories.
	Let there be a strong action of $\X$ on $\C$.
	Then the semidirect product $\X \ltimes \C$ is a right closed monoidal category.
\label{conjecture:right-closed} \end{conjecture}
However, this conjecture is false, as we will now show.
Consider Example (\ref{example:counterexample-closed}); we will show that tensoring on the right by $\langle F, 1 \rangle$ does not preserve colimits.
Denote by $D_t$ the generalised metric space consisting of two points separated by a distance $t$ in each direction.
There is an obvious morphism $D_s \rightarrow D_t$ whenever $s \geq t$.
Consider the following diagram.
\[ \langle T, D_{\frac{1}{1}} \rangle \rightarrow \langle T, D_{\frac{1}{2}} \rangle \rightarrow \langle T, D_{\frac{1}{3}} \rangle \rightarrow \langle T, D_{\frac{1}{4}} \rangle \rightarrow \dotsb \]
The colimit of this diagram is the object $\langle T, D_0 \rangle$.
The image of this diagram under the functor $(- \tensor \langle F, 1 \rangle)$ is the following diagram.
\[ \langle F, D_\infty \rangle \rightarrow \langle F, D_\infty \rangle \rightarrow \langle F, D_\infty \rangle \rightarrow \langle F, D_\infty \rangle \rightarrow \dotsb \]
The colimit of this diagram is the object $\langle F, D_\infty \rangle$.
Thus, if tensoring on the right by $\langle F, 1 \rangle$ were to preserve colimits, we would expect an isomorphism of the following form.
\[ \langle T, D_0 \rangle \tensor \langle F, 1 \rangle \iso \langle F, D_\infty \rangle \]
However, the left hand side of this equation evaluates as follows.
\[ \langle T, D_0 \rangle \tensor \langle F, 1 \rangle = \langle F, D_0 \rangle \]
Thus, no such isomorphism exists, so tensoring on the right by $\langle F, 1 \rangle$ cannot preserve colimits, and so $\X \ltimes \C$ cannot be right closed.
This provides a counterexample to Conjecture (\ref{conjecture:right-closed}).

However, the following weaker result does hold.
\begin{theorem}
	Let $\X$ and $\C$ be right closed monoidal categories.
	Let there be a strong action of $\X$ on $\C$ such that each strong monoidal endofunctor $(-)^X$ has a right adjoint, denoted $(-)_X$.
	Then the semidirect product $\X \ltimes \C$ is a right closed monoidal category, with internal hom defined as follows.
	\[ [\langle Y, B \rangle, \langle Z, C \rangle] = \langle [Y, Z], [B, C]_Y \rangle \]
\label{theorem:right-closed} \end{theorem}
\begin{proof}
	We have the following natural isomorphism of hom sets.
	\begin{align*}
		(\X \ltimes \C) (\langle X, A \rangle \tensor \langle Y, B \rangle, \langle Z, C \rangle)
		&= (\X \ltimes \C) (\langle X \tensor Y, A^Y \tensor B \rangle, \langle Z, C \rangle) \\
		&= \X (X \tensor Y, Z) \times \C (A^Y \tensor B, C) \\
		&\iso \X (X, [Y, Z]) \times \C (A^Y, [B, C]) \\
		&\iso \X (X, [Y, Z]) \times \C (A, [B, C]_Y) \\
		&= (\X \ltimes \C) (\langle X, A \rangle, \langle [Y, Z], [B, C]_Y \rangle) \\
		&= (\X \ltimes \C) (\langle X, A \rangle, [\langle Y, B \rangle, \langle Z, C \rangle])
	\end{align*}
\end{proof}
\begin{corollary}
	Let $\X$ be a left autonomous monoidal category.
	Let $\C$ be a right closed monoidal category.
	Let there be a strong action of $\X$ on $\C$.
	Then the semidirect product $\X \ltimes \C$ is a right closed monoidal category, with internal hom defined as follows.
	\[ [\langle Y, B \rangle, \langle Z, C \rangle] = \langle Z \tensor \duall{Y}, [B, C]^\duall{Y} \rangle \]
\label{corollary:right-closed} \end{corollary}
\begin{proof}
	Strong monoidal functors preserve duals.
	In particular, the strong monoidal functor $\Gamma \colon \X \rightarrow [\C, \C]$ preserves duals.
	This means that, for each object $X$ in $\X$, the strong monoidal endofunctor $(-)^X$ has a right adjoint, $(-)^\duall{X}$.
	See \cite{JoyalStreet1993} for details.
\end{proof}

We will now give an example of a right closed semidirect product monoidal category.

\begin{example}
	This example involves generalised metric spaces, as described by Lawvere \cite{Lawvere1973}.
	Let $\X$ be $[0, \infty)$ the category whose objects are the non-negative real numbers, with a unique morphism $x \rightarrow y$ if and only if $x \geq y$.
	This category has a closed symmetric monoidal structure, with tensor product given by addition and internal hom given by truncated subtraction, defined as follows.
	\[ x \tensor y = x + y \qquad [x, y] = \max (y - x, 0) \]
	Let $\C$ be the closed symmetric monoidal category of generalised metric spaces.
	Then there is a strong action of $\X$ on $\C$, given by scaling, in which the underlying set of $M^x$ is the same as the underlying set of $M$, but with a new scaled metric, defined as follows.
	\[ M^x (m, m') = M (m, m') \cdot e^x \]
	Each functor $(-)^x$ has a right adjoint (in fact, an inverse), $(-)_x$, in which the underlying set of $M_x$ is the same as the underlying set of $M$, but with a new scaled metric, defined as follows.
	\[ M_x (m, m') = M (m, m') \cdot e^{-x} \]
	The resulting semidirect product $\X \ltimes \C$ is a right closed monoidal structure on $[0, \infty) \times \C$, with tensor product and internal hom defined as follows.
	\[ \langle x, M \rangle \tensor \langle y, N \rangle = \langle x + y, M^y \tensor N \rangle \]
	\[ [\langle y, N \rangle, \langle z, P \rangle] = \langle \max (z - y, 0), [N, P]_y \rangle \]
	The metric on the generalised metric space $M^y \tensor N$ with underlying set $M \times N$ is given as follows.
	\[ (M^y \tensor N) ((m, n), (m', n')) = M (m, m') \cdot e^y + N (n, n') \]
	The metric on the generalised metric space $[N, P]_y$ with underlying set $\C (N, P)$ is given as follows.
	\[ [N, P]_y (f, g) = \sup_{n \in N} P (f (n), g (n)) \cdot e^{-y} \]
\end{example}

\section{Left Closed Monoidal Categories} \label{section:left-closed}

In this section, we will discuss the possibility of a semidirect product monoidal category being left closed.
We will show that the analogue of Theorem (\ref{theorem:left-autonomous}) for left closed monoidal categories is not true.
We will then produce some examples of semidirect product monoidal categories which are left closed but not right closed.

Throughout this section, we will use the term `left closed monoidal category' to mean a monoidal category $\C$ in which tensoring on the left has a right adjoint.
We will denote the internal hom as follows.
\[ [-, -] \colon \C^\op \times \C \rightarrow \C \]
Using this notation, the hom-tensor adjunction is the following natural isomorphism of hom sets.
\[ \C (A \tensor B, C) \iso \C (B, [A, C]) \]

We might hope for the following analogue of Theorem (\ref{theorem:left-autonomous}) for left closed monoidal categories to hold.
\begin{conjecture}
	Let $\X$ and $\C$ be left closed monoidal categories.
	Let there be a strong action of $\X$ on $\C$.
	Then the semidirect product $\X \ltimes \C$ is a left closed monoidal category.
\label{conjecture:left-closed} \end{conjecture}
However, this conjecture is false, as we will now show.
Consider Example (\ref{example:counterexample-closed}); we will show that tensoring on the left by an arbitrary object $\langle x, M \rangle$ does not preserve coproducts.
Consider the following coproduct diagram.
\[ \langle F, 1 \rangle \qquad \langle T, 0 \rangle \]
The colimit of this diagram is the object $\langle T, 1 \rangle$.
The image of this diagram under the functor $(\langle x, M \rangle \tensor -)$ is the following diagram.
\[ \langle F, M^F \rangle \qquad \langle x, 0 \rangle \]
The colimit of this diagram is the object $\langle x, M^F \rangle$.
Thus, if tensoring on the left by $\langle x, M \rangle$ were to preserve colimits, we would expect an isomorphism of the following form.
\[ \langle x, M \rangle \tensor \langle T, 1 \rangle \iso \langle x, M^F \rangle \]
However, the left hand side of this equation evaluates as follows.
\[ \langle x, M \rangle \tensor \langle T, 1 \rangle = \langle x, M^T \rangle \]
Thus, no such isomorphism exists, so tensoring on the left by $\langle x, M \rangle$ cannot preserve colimits, and so $\X \ltimes \C$ cannot be left closed.
This provides a counterexample to Conjecture (\ref{conjecture:left-closed}).

We will now produce some examples of semidirect product monoidal categories which are left closed but not right closed.

Let $\X$ be a left closed monoidal category.
Let $\C$ be a cocartesian monoidal category.
Let there be a strong action of $\X$ on $\C$.
As before, there is a semidirect product monoidal category, $\X \ltimes \C$, with tensor product defined as follows.
\[ \langle X, B \rangle \tensor \langle Y, C \rangle = \langle X \tensor Y, B^Y + C \rangle \]
In addition to this, assume that $\X$ has finite products and that there is another functor
\[ (- \rhd -) \colon \C^\op \times \C \rightarrow \X \]
and a natural isomorphism of hom sets of the following form.
\[ \C (B^X, C) \iso \X (X, B \rhd C) \]
We claim that, under these conditions, the monoidal category $\X \ltimes \C$ is left closed.

\begin{theorem}
	The monoidal category $\X \ltimes \C$ is left closed, with internal hom defined as follows.
	\[ [\langle X, A \rangle, \langle Z, C \rangle] = \langle [X, Z] \times (A \rhd C), C \rangle \]
\end{theorem}
\begin{proof}
	We have the following natural isomorphism of hom sets.
	\begin{align*}
		(\X \ltimes \C) (\langle X, A \rangle \tensor \langle Y, B \rangle, \langle Z, C \rangle)
		&= (\X \ltimes \C) (\langle X \tensor Y, A^Y + B \rangle, \langle Z, C \rangle) \\
		&= \X (X \tensor Y, Z) \times \C (A^Y + B, C) \\
		&\iso \X (X \tensor Y, Z) \times \C (A^Y, C) \times \C (B, C) \\
		&\iso \X (Y, [X, Z]) \times \X (Y, A \rhd C) \times \C (B, C) \\
		&\iso \X (Y, [X, Z] \times (A \rhd C)) \times \C (B, C) \\
		&= (\X \ltimes \C) (\langle Y, B \rangle, \langle [X, Z] \times (A \rhd C), C \rangle) \\
		&= (\X \ltimes \C) (\langle Y, B \rangle, [\langle X, A \rangle, \langle Z, C \rangle])
	\end{align*}
\end{proof}

However, the monoidal category $\X \ltimes \C$ is not, in general, right closed, as we will now show.
In any right closed monoidal category, tensoring on the right with a fixed object has a right adjoint, and thus preserves the initial object.
So, if $\X \ltimes \C$ were a right closed monoidal category, then we would necessarily have isomorphisms of the following form.
\[ 0_{\X \ltimes \C} \tensor \langle X, C \rangle \iso 0_{\X \ltimes \C} \]
Evaluating each side of this equation gives the following.
\[ \langle 0_\X \tensor X, C \rangle \iso \langle 0_\X, 0_\C \rangle \]
So, we would necessarily have isomorphisms of the following forms.
\[ 0_\X \tensor X \iso 0_\X \qquad C \iso 0_\C \]
The first may exist, but the second will not, unless $\C$ is trivial.

We will now give some examples of semidirect product monoidal categories which are left closed but not right closed.

\begin{example}
	Let $\X$ be a left closed monoidal category with finite products and finite coproducts, in which the tensor product preserves coproducts in both variables.
	Let $\C$ be $\X$, considered as a cocartesian monoidal category.
	Then there is a strong action of $\X$ on $\C$ using the original tensor product, as follows.
	\[ B^X = B \tensor X \]
	Let $(- \rhd -)$ be the original internal hom, as follows.
	\[ B \rhd C = [B, C] \]
	The isomorphism of hom sets
	\[ \C (B^X, C) \iso \X (X, B \rhd C) \]
	is then just the usual hom-tensor adjunction.
	The resulting semidirect product $\X \ltimes \C$ is a left closed monoidal structure on $\X \times \X$, with tensor product and internal hom defined as follows.
	\[ \langle X, A \rangle \tensor \langle Y, B \rangle = \langle X \tensor Y, (A \tensor Y) + B \rangle \]
	\[ [\langle X, A \rangle, \langle Z, C \rangle] = \langle [X, Z] \times [A, C], C \rangle \]
\end{example}

\begin{example}
	Let $\X$ be $\Set$, the category of sets.
	Let $\C$ be a category with small coproducts, considered as a cocartesian monoidal category.
	Then there is a strong action of $\X$ on $\C$ by copowers, as follows.
	\[ C^X = \coprod_{x \in X} C \]
	The notation we have been using so far agrees with the notation usually used for powers, rather than copowers; this is unfortunate, but hopefully not too confusing.
	Let $(- \rhd -)$ be the usual hom-functor, as follows.
	\[ B \rhd C = \C (B, C) \]
	The isomorphism of hom sets
	\[ \C (B^X, C) \iso \X (X, B \rhd C) \]
	is then just the universal property of the copower.
	The resulting semidirect product $\X \ltimes \C$ is a left closed monoidal structure on $\Set \times \C$, with tensor product and internal hom defined as follows.
	\[ \langle X, A \rangle \tensor \langle Y, B \rangle = \langle X \times Y, (\coprod_{y \in Y} A) + B \rangle \]
	\[ [\langle X, A \rangle, \langle Z, C \rangle] = \langle \Set (X, Z) \times \C (A, C), C \rangle \]
\end{example}

\begin{example}
	As a specific case of the previous example, let $\C$ be a complete lattice, considered as a preorder.
	In this category, the coproduct of $a$ and $b$ is their join, or least upper bound, denoted $a \vee b$.
	The resulting semidirect product $\X \ltimes \C$ is a left closed monoidal structure on $\Set \times \C$, with tensor product and internal hom defined as follows.
	\[ \langle X, a \rangle \tensor \langle Y, b \rangle =
	\begin{cases}
		\langle X \times Y, a \vee b \rangle & \text{if}~ Y ~\text{is non-empty} \\
		\langle \emptyset, b \rangle & \text{if}~ Y ~\text{is empty}
	\end{cases} \]
	\[ [\langle X, a \rangle, \langle Z, c \rangle] =
	\begin{cases}
		\langle \Set (X, Z), c \rangle & \text{if}~ a \leq c \\
		\langle \emptyset, c \rangle & \text{if}~ a > c
	\end{cases} \]
\end{example}


\bibliographystyle{plain}

\bibliography{bibliography}

\end{document}